\documentclass[12pt,lenq]{amsart}
\usepackage{amsmath}
\usepackage{amscd}
\usepackage{amssymb}
\usepackage{amsbsy}
\usepackage{amsfonts}
\usepackage{latexsym}
\usepackage{graphics}
\usepackage{amsmath,amscd,latexsym}
\usepackage{multirow}
\usepackage{array}
\usepackage{paralist}
\usepackage{titletoc}

\theoremstyle{plain}
\newtheorem{theorem}{Theorem}[section]
\newtheorem{proposition}[theorem]{Proposition}
\newtheorem{lemma}[theorem]{Lemma}
\newtheorem{corollary}[theorem]{Corollary}
\newtheorem{remark}[theorem]{Remark}
\newtheorem{definition}[theorem]{Definition}

\newtheorem{main theorem}[theorem]{Main Theorem}

\newtheorem{question}[theorem]{Question}
\newtheorem{convention}[theorem]{Convention}

\newcommand{\interior}{\operatorname{int}}

\newcommand{\ZZ}{\mathbb{Z}}
\newcommand{\QQ}{\mathbb{Q}}
\newcommand{\RR}{\mathbb{R}}

\newcommand{\HH}{\mathbb{H}}
\newcommand{\QQQ}{\hat{\mathbb{Q}}}
\newcommand{\RRR}{\hat{\mathbb{R}}}

\newcommand{\PConway}{\mbox{\boldmath$S$}}

\newcommand{\rtangle}[1]{(B^3,t({#1}))}

\newcommand{\DD}{\mathcal{D}}
\newcommand{\RGPC}[2]{\Gamma({#1};{#2})}

\newcommand{\RGP}[1]{\Gamma_{#1}}

\newcommand{\Hecke}{\mbox{$G$}}

\newcommand{\orbs}{\mbox{\boldmath$S$}}

\newcommand{\svert}{\,|\,}

\newcommand{\llangle}{\langle\langle}
\newcommand{\rrangle}{\rangle\rangle}

\newcommand{\lp}{(\hskip -0.07cm (}
\newcommand{\rp}{)\hskip -0.07cm )}

\begin{document}

\title{Epimorphisms from 2-bridge link groups onto Heckoid groups (II)}

\author{Donghi Lee}
\address{Department of Mathematics\\
Pusan National University \\
San-30 Jangjeon-Dong, Geumjung-Gu, Pusan, 609-735, Republic of Korea}
\email{donghi@pusan.ac.kr}

\author{Makoto Sakuma}
\address{Department of Mathematics\\
Graduate School of Science\\
Hiroshima University\\
Higashi-Hiroshima, 739-8526, Japan}
\email{sakuma@math.sci.hiroshima-u.ac.jp}

\subjclass[2010]{Primary 20F06, 57M25 \\
\indent {The first author was supported by the Research Fund Program
of Research Institute for Basic Sciences, Pusan National University,
Korea, 2009, Project No.\,RIBS-PNU-2009-103.
The second author was supported by JSPS Grants-in-Aid 22340013.}}

\begin{abstract}
In Part I of this series of papers, we made
Riley's definition of Heckoid groups for $2$-bridge links explicit,
and gave a systematic construction
of epimorphisms from $2$-bridge link groups onto
Heckoid groups, generalizing Riley's construction.
In this paper, we give a complete characterization of
upper-meridian-pair-preserving epimorphisms
from $2$-bridge link groups onto even Heckoid groups,
by proving that they are exactly the epimorphisms
obtained by the systematic construction.
\end{abstract}
\maketitle

\begin{center}
{\it In honour of J. Hyam Rubinstein
and his contribution to mathematics}
\end{center}

\section{Introduction}

Let $K(r)$ be the $2$-bridge link of slope $r \in \QQ$
and let $n$ be an integer or a half-integer greater than $1$.
In \cite{lee_sakuma_6}, following Riley's work ~\cite{Riley2},
we introduced the {\it Heckoid group $\Hecke(r;n)$ of index $n$ for $K(r)$}
as the orbifold fundamental group
of the {\it Heckoid orbifold $\orbs(r;n)$ of index $n$ for $K(r)$}.
According to whether $n$ is an integer or a non-integral half-integer,
the Heckoid group $\Hecke(r;n)$ and the Heckoid orbifold $\orbs(r;n)$
are said to be {\it even} or {\it odd}.
The even Heckoid orbifold $\orbs(r;n)$
is the $3$-orbifold such that
\begin{enumerate}[\rm (i)]
\item
the underlying space $|\orbs(r;n)|$ is the exterior,
$E(K(r))=S^3-\interior N(K(r))$, of $K(r)$, and
\item
the singular set is the lower tunnel of $K(r)$,
where the index of the singularity is $n$.
\end{enumerate}
For a description of odd Heckoid orbifolds,
see \cite[Proposition ~5.3]{lee_sakuma_6}.

In \cite[Theorem ~2.3]{lee_sakuma_6},
we gave a systematic construction
of upper-meridian-pair-preserving epimorphisms
from $2$-bridge link groups onto
Heckoid groups,
generalizing Riley's construction in \cite{Riley2}.

The main purpose of this paper is to describe all
upper-meridian-pair-preserving epimorphisms
from $2$-bridge link groups onto {\it even} Heckoid groups
(Theorem ~\ref{thm:epimorophism2}).
The theorem says that all such epimorphisms are contained in
those constructed in \cite[Theorem ~2.3]{lee_sakuma_6}.
To prove this result, we determine
those essential simple loops on a $2$-bridge sphere
in an even Heckoid orbifold $\orbs(r;n)$
which are null-homotopic in $\orbs(r;n)$ (Theorem ~\ref{thm:null-homotopy}).
These results form an analogy of \cite[Main Theorem ~2.4]{lee_sakuma},
which describes all upper-meridian-pair-preserving epimorphisms
between $2$-bridge link groups,
and that of \cite[Main Theorem ~2.3]{lee_sakuma},
which gives a complete characterization of
those essential simple loops on a $2$-bridge sphere
in a $2$-bridge link complement
which are null-homotopic in the link complement.
As in \cite{lee_sakuma}, the key tool is small cancellation theory,
applied to two-generator and one-relator presentations
of even Heckoid groups.

This paper is organized as follows.
In Section ~\ref{statements},
we describe the main results.
In Section ~\ref{group_presentation}, we introduce
a two-generator and one-relator presentation of an even Heckoid group, and
review basic facts concerning its single relator established in \cite{lee_sakuma}.
In Section ~\ref{sec:small_cancellation_theory},
we apply small cancellation theory to the two-generator and one-relator presentations
of even Heckoid groups.
In Section ~\ref{sec:proof_of_the_theorem},
we prove Theorem ~\ref{thm:null-homotopy}.

\section{Main results}
\label{statements}

We quickly recall notation and basic facts introduced in \cite{lee_sakuma_6}.
The {\it Conway sphere} $\PConway$ is the 4-times punctured sphere
which is obtained as the quotient of $\RR^2-\ZZ^2$
by the group generated by the $\pi$-rotations around
the points in $\ZZ^2$.
For each $s \in \QQQ:=\QQ\cup\{\infty\}$,
let $\alpha_s$ be the simple loop in $\PConway$
obtained as the projection of a line in $\RR^2-\ZZ^2$
of slope $s$.
We call $s$ the {\it slope} of the simple loop $\alpha_s$.

For each $r\in \QQQ$,
the {\it $2$-bridge link $K(r)$ of slope $r$}
is the sum of the rational tangle
$\rtangle{\infty}$ of slope $\infty$ and
the rational tangle $\rtangle{r}$ of slope $r$.
Recall that $\partial(B^3-t(\infty))$ and $\partial(B^3-t(r))$
are identified with $\PConway$
so that $\alpha_{\infty}$ and $\alpha_r$
bound disks in $B^3-t(\infty)$ and $B^3-t(r)$, respectively.
By van-Kampen's theorem, the link group $G(K(r))=\pi_1(S^3-K(r))$ is obtained as follows:
\[
G(K(r))=\pi_1(S^3-K(r))
\cong \pi_1(\PConway)/ \llangle\alpha_{\infty},\alpha_r\rrangle
\cong \pi_1(B^3-t(\infty))/\llangle\alpha_r\rrangle.
\]
We call the image in the link group
of the ``meridian pair'' of $\pi_1(B^3-t(\infty))$
the {\it upper meridian pair}.

If $r$ is a rational number and $n\ge 2$ is an integer,
then by the description of the even Heckoid orbifold $\orbs(r;n)$ in the introduction,
the even Hekoid group $\Hecke(r;n)=\pi_1(\orbs(r;n))$ is identified with
\[
\Hecke(r;n)
\cong\pi_1(\PConway)/ \llangle\alpha_{\infty},\alpha_r^n\rrangle
\cong \pi_1(B^3-t(\infty))/\llangle\alpha_r^n\rrangle.
\]
In particular,
the even Heckoid group $\Hecke(r;n)$ is a two-generator and one-relator group.
We call the image in $\Hecke(r;n)$
of the meridian pair of $\pi_1(B^3-t(\infty))$
the {\it upper meridian pair}.

This paper and its sequel ~\cite{lee_sakuma_8}
are concerned with the following natural
question, which is an analogy of \cite[Question ~1.1]{lee_sakuma_0}
that is completely solved in the series of papers
\cite{lee_sakuma, lee_sakuma_2, lee_sakuma_3, lee_sakuma_4}
and applied in \cite{lee_sakuma_5}.

\begin{question}
\label{question1}
{\rm
For $r$ a rational number and $n$ an integer or a half-integer
greater than $1$,
consider the Heckoid group $\Hecke(r;n)$ of index $n$ for the $2$-bridge link $K(r)$.
\begin{enumerate}[\indent \rm (1)]
\item
Which essential simple loop $\alpha_s$ on $\PConway$
determines the trivial element of $\Hecke(r;n)$?
\item
For two distinct essential simple loops $\alpha_s$ and $\alpha_{s'}$ on $\PConway$,
when do they determine the same conjugacy class in $\Hecke(r;n)$?
\end{enumerate}
}
\end{question}

In \cite[Theorem ~2.4]{lee_sakuma_6},
we gave a certain sufficient condition for each of the questions.
In this paper, we prove that, for even Heckoid groups,
the sufficient condition for (1) is actually a necessary and sufficient condition.
This enables us to describe all upper-meridian-pair-preserving
epimorphisms from $2$-bridge link groups onto even Heckoid groups.

Let $\DD$ be the {\it Farey tessellation}
of the upper half plane $\HH^2$.
Then $\QQQ$ is identified with the set of the ideal vertices of $\DD$.
Let $\RGP{\infty}$ be the group of automorphisms of
$\DD$ generated by reflections in the edges of $\DD$ with an endpoint $\infty$.
For $r$ a rational number
and $n$ an integer or a half-integer greater than $1$,
let $C_r(2n)$ be the group of automorphisms of $\DD$ generated
by the parabolic transformation, centered on the vertex $r$,
by $2n$ units in the clockwise direction,
and let $\RGPC{r}{n}$ be the group generated by $\RGP{\infty}$ and $C_r(2n)$.
Suppose that $r$ is not an integer, i.e., $K(r)$ is not a trivial knot.
Then $\RGPC{r}{n}$ is the free product $\RGP{\infty}*C_r(2n)$
having a fundamental domain, $R$, shown in Figure ~\ref{fig.fd_orbifold}.
Here, $R$ is obtained as the intersection of fundamental domains
for $\RGP{\infty}$ and $C_r(2n)$, and so
$R$ is bounded by the following two pairs of Farey edges:
\begin{enumerate}[\indent \rm (1)]
\item
the pair of adjacent Farey edges with an endpoint $\infty$
which cuts off a region in $\bar\HH^2$ containing $r$, and
\item
a pair of Farey edges with an endpoint $r$
which cuts off a region in $\bar\HH^2$ containing $\infty$ such that
one edge is the image of the other by a generator of $C_r(2n)$.
\end{enumerate}

Let $\bar{I}(r;n)$ be the union of
two closed intervals
in $\partial \HH^2=\RRR$
obtained as the intersection of the closure of $R$ and $\partial \HH^2$.
(In the special case when $r\equiv \pm1/p \pmod{\ZZ}$
for some integer $p>1$,
one of the intervals may be degenerated to a single point.)
Note that there is a pair $\{r_1, r_2\}$
of boundary points of $\bar{I}(r;n)$
such that $r_2$ is the image of $r_1$ by a generator of $C_r(2n)$.
Set $I(r;n):=\bar{I}(r;n) -\{r_i\}$ with $i=1$ or $2$.
Note that
$I(r;n)$ is the disjoint union of a closed interval and a half-open interval,
except for the special case when $r\equiv \pm 1/p \pmod{\ZZ}$.

\begin{figure}[htbp]
\begin{center}
\includegraphics{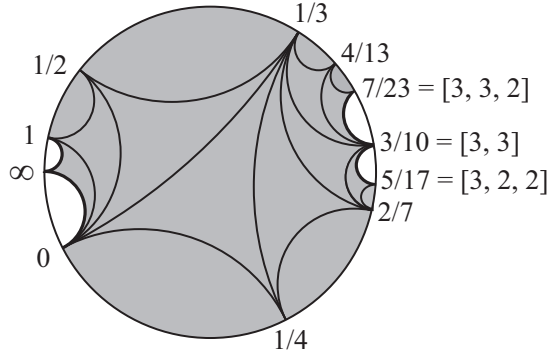}
\end{center}
\caption{\label{fig.fd_orbifold}
A fundamental domain of $\RGPC{r}{n}$ in the
Farey tessellation (the shaded domain) for $r=3/10=${$\cfrac{1}{3+\cfrac{1}{3}}$}\,$=:[3,3]$ and $n=2$.
In this case, $\bar{I}(r;n)=[0,5/17] \cup [7/23,1]$.}
\end{figure}

Then we obtain the following refinement of
\cite[Theorem ~2.4]{lee_sakuma_6}.

\begin{theorem}
\label{thm:fundametal_domain}
Suppose that $r$ is a non-integral rational number and
that $n$ is an integer or a half-integer greater than $1$.
Then, for any $s\in\QQQ$, there is a unique rational number
$s_0\in I(r;n) \cup \{\infty, r\}$
such that $s$ is contained in the $\RGPC{r}{n}$-orbit of $s_0$.
Moreover the conjugacy classes
$\alpha_s$ and $\alpha_{s_0}$ in $\Hecke(r;n)$ are equal.
In particular, if $s_0=\infty$, then $\alpha_s$ is the trivial conjugacy class
in $\Hecke(r;n)$.
\end{theorem}

In fact, the first assertion is proved as in \cite[Lemma ~7.1]{lee_sakuma}
by using the fact
that $R$ is a fundamental domain for the action of $\RGPC{r}{n}$ on $\HH^2$.
The remaining assertions are nothing other than \cite[Theorem ~2.4]{lee_sakuma_6}.

The following main theorem shows that
the converse to the last statement in Theorem ~\ref{thm:fundametal_domain}
holds for {\it even} Heckoid groups.

\begin{theorem}
\label{thm:null-homotopy}
Suppose that $r$ is a non-integral rational number and
that $n$ is an integer greater than $1$.
Then $\alpha_s$ represents the trivial element of $\Hecke(r;n)$
if and only if $s$ belongs to the $\RGPC{r}{n}$-orbit of $\infty$.
In other words, if $s\in I(r;n) \cup \{r\}$, then
$\alpha_s$ does not represent the trivial element of $\Hecke(r;n)$.
\end{theorem}

Arguing as in \cite[Proof of Theorem ~2.3]{lee_sakuma_6},
we see that
the above theorem implies the following theorem,
which says that the converse to \cite[Theorem ~2.3]{lee_sakuma_6}
holds for even Heckoid groups.

\begin{theorem}
\label{thm:epimorophism2}
Suppose that $r$ is a non-integral rational number and
that $n$ is an integer greater than $1$.
Then there is an upper-meridian-pair-preserving epimorphism
from $G(K(s))$ to $\Hecke(r;n)$
if and only if $s$ or $s+1$ belongs to the
$\RGPC{r}{n}$-orbit of
$\infty$.
\end{theorem}

\begin{remark}
\label{remark:remaining-case}
{\rm
(1) When $r$ is an integer,
the Heckoid group $\Hecke(r;n)\cong \Hecke(0;n)$ is isomorphic to
the subgroup $\langle P, SPS^{-1}\rangle$
of the classical {\it Hecke group}
$\langle P, S\rangle$ introduced in \cite{Hecke}, where
\[
P=\begin{pmatrix}
1 & \quad 2\cos\frac{\pi}{2n}\\
0 & \quad 1
\end{pmatrix},
\quad
S=
\begin{pmatrix}
0 & \quad 1\\
-1 & \quad 0
\end{pmatrix}.
\]
Moreover, the group $\Gamma(0;n)$ is the free product of
three cyclic groups of order $2$ generated by the reflections
in the Farey edges $\langle \infty,0\rangle$ and $\langle \infty,1\rangle$
and the geodesic $\overline{1,1/n}$.
(The last geodesic is a Farey edge if $n$ is an integer,
whereas it bisects a pair of adjacent Farey triangles if $n$ is a non-integral half-integer.)
The region of $\HH^2$ bounded by these three geodesics is a fundamental domain
for the action of $\Gamma(0;n)$ on $\HH^2$.
It is easy to see that Theorem ~\ref{thm:fundametal_domain}
continues to be valid when $r$ is an integer,
provided that we set $I(0;n):=[1/n,n]$.
It is plausible that Theorems ~\ref{thm:null-homotopy} and \ref{thm:epimorophism2}
are also valid even when $r$ is an integer.
However, we cannot directly apply the arguments of this paper, and
this case will be treated elsewhere.

(2) It is natural to expect that
Theorems ~\ref{thm:null-homotopy} and
\ref{thm:epimorophism2} also hold for odd Heckoid groups.
However, we do not know how to treat these groups at this moment,
because they are not one-relator groups by
\cite[Proposition ~6.7]{lee_sakuma_6}.
}
\end{remark}

\section{Presentations of even Heckoid groups and review of basic facts from \cite{lee_sakuma}}
\label{group_presentation}

In the remainder of this paper, we restrict our attention
to the {\it even} Heckoid groups $\Hecke(r;n)$.
Thus $n$ denotes an integer with $n\ge 2$.
In order to describe the two-generator and one-relator presentations of even Heckoid groups
to which we apply small cancellation theory,
recall that
\[
\Hecke(r;n)\cong
\pi_1(\PConway)/ \llangle\alpha_{\infty},\alpha_r^n\rrangle
\cong
\pi_1(B^3-t(\infty))/ \llangle\alpha_r^n\rrangle.
\]
Let $\{a,b\}$ be the standard meridian generator pair of $\pi_1(B^3-t(\infty), x_0)$
as described in \cite[Section ~3]{lee_sakuma}
(see also \cite[Section ~5]{lee_sakuma_0}).
Then $\pi_1(B^3-t(\infty))$ is identified with the free group $F(a,b)$.
For the rational number $r=q/p$, where $p$ and $q$ are relatively prime positive integers,
let $u_r$ be the word in $\{a,b\}$ obtained as follows.
(For a geometric description, see \cite[Section ~5]{lee_sakuma_0}.)
Set $\epsilon_i = (-1)^{\lfloor iq/p \rfloor}$,
where $\lfloor x \rfloor$ is the greatest integer not exceeding $x$.
\begin{enumerate}[\indent \rm (1)]
\item If $p$ is odd, then
\[u_{q/p}=a\hat{u}_{q/p}b^{(-1)^q}\hat{u}_{q/p}^{-1},\]
where
$\hat{u}_{q/p} = b^{\epsilon_1} a^{\epsilon_2} \cdots b^{\epsilon_{p-2}} a^{\epsilon_{p-1}}$.
\item If $p$ is even, then
\[u_{q/p}=a\hat{u}_{q/p}a^{-1}\hat{u}_{q/p}^{-1},\]
where
$\hat{u}_{q/p} = b^{\epsilon_1} a^{\epsilon_2} \cdots a^{\epsilon_{p-2}} b^{\epsilon_{p-1}}$.
\end{enumerate}
Then $u_r\in F(a,b)\cong\pi_1(B^3-t(\infty))$
is represented by the simple loop $\alpha_r$,
and we obtain the following two-generator and one-relator presentation
of the even Heckoid group $\Hecke(r;n)$,
which is used throughout the remainder of this paper:
\[
\Hecke(r;n) \cong\pi_1(B^3-t(\infty))/\llangle \alpha_r^n\rrangle
\cong \langle a, b \svert u_r^n \rangle.
\]

We recall the definition of the sequences $S(r)$ and $T(r)$ and the cyclic sequences
$CS(r)$ and $CT(r)$ of slope $r$ defined in \cite{lee_sakuma},
all of which are read from the word $u_r$ defined above,
and review several important properties of these sequences
from \cite{lee_sakuma} so that we can adopt small cancellation theory
in the succeeding section.
To this end, we fix some definitions and notation.
Let $X$ be a set.
By a {\it word} in $X$, we mean a finite sequence
$x_1^{\epsilon_1}x_2^{\epsilon_2}\cdots x_t^{\epsilon_t}$
where $x_i\in X$ and $\epsilon_i=\pm 1$.
Here we call $x_i^{\epsilon_i}$ the {\it $i$-th letter} of the word.
For two words $u, v$ in $X$, by
$u \equiv v$ we denote the {\it visual equality} of $u$ and
$v$, meaning that if $u=x_1^{\epsilon_1} \cdots x_t^{\epsilon_t}$
and $v=y_1^{\delta_1} \cdots y_m^{\delta_m}$ ($x_i, y_j \in X$; $\epsilon_i, \delta_j=\pm 1$),
then $t=m$ and $x_i=y_i$ and $\epsilon_i=\delta_i$ for each $i=1, \dots, t$.
For example, two words $x_1x_2x_2^{-1}x_3$ and $x_1x_3$ ($x_i \in X$) are {\it not} visually equal,
though $x_1x_2x_2^{-1}x_3$ and $x_1x_3$ are equal as elements of the free group with basis $X$.
The length of a word $v$ is denoted by $|v|$.
A word $v$ in $X$ is said to be {\it reduced} if $v$ does not contain $xx^{-1}$ or $x^{-1}x$ for any $x \in X$.
A word is said to be {\it cyclically reduced}
if all its cyclic permutations are reduced.
A {\it cyclic word} is defined to be the set of all cyclic permutations of a
cyclically reduced word. By $(v)$ we denote the cyclic word associated with a
cyclically reduced word $v$.
Also by $(u) \equiv (v)$ we mean the {\it visual equality} of two cyclic words
$(u)$ and $(v)$. In fact, $(u) \equiv (v)$ if and only if $v$ is visually a cyclic shift
of $u$.

\begin{definition}
\label{def:alternating}
{\rm (1) Let $v$ be a reduced word in
$\{a,b\}$. Decompose $v$ into
\[
v \equiv v_1 v_2 \cdots v_t,
\]
where, for each $i=1, \dots, t-1$, all letters in $v_i$ have positive (resp., negative) exponents,
and all letters in $v_{i+1}$ have negative (resp., positive) exponents.
Then the sequence of positive integers
$S(v):=(|v_1|, |v_2|, \dots, |v_t|)$ is called the {\it $S$-sequence of $v$}.

(2) Let $(v)$ be a cyclic word in
$\{a, b\}$. Decompose $(v)$ into
\[
(v) \equiv (v_1 v_2 \cdots v_t),
\]
where all letters in $v_i$ have positive (resp., negative) exponents,
and all letters in $v_{i+1}$ have negative (resp., positive) exponents (taking
subindices modulo $t$). Then the {\it cyclic} sequence of positive integers
$CS(v):=\lp |v_1|, |v_2|, \dots, |v_t| \rp$ is called
the {\it cyclic $S$-sequence of $(v)$}.
Here the double parentheses denote that the sequence is considered modulo
cyclic permutations.

(3) A reduced word $v$ in $\{a,b\}$ is said to be {\it alternating}
if $a^{\pm 1}$ and $b^{\pm 1}$ appear in $v$ alternately,
i.e., neither $a^{\pm2}$ nor $b^{\pm2}$ appears in $v$.
A cyclic word $(v)$ is said to be {\it alternating}
if all cyclic permutations of $v$ are alternating.
In the latter case, we also say that $v$ is {\it cyclically alternating}.
}
\end{definition}

\begin{definition}
\label{def4.1(3)}
{\rm
For a rational number $r$ with $0<r\le 1$,
let $u_r$ be the word defined in the beginning of this section.
Then the symbol $S(r)$ (resp., $CS(r)$) denotes the
$S$-sequence $S(u_r)$ of $u_r$
(resp., cyclic $S$-sequence $CS(u_r)$ of $(u_r)$), which is called
the {\it S-sequence of slope $r$}
(resp., the {\it cyclic S-sequence of slope $r$}).}
\end{definition}

In the remainder of this section, we suppose that $r$ is a rational number
with $0<r\le1$, and write $r$ as a continued fraction expansion:
\begin{center}
\begin{picture}(230,70)
\put(0,48){$\displaystyle{
r=[m_1,m_2, \dots,m_k]:=
\cfrac{1}{m_1+
\cfrac{1}{ \raisebox{-5pt}[0pt][0pt]{$m_2 \, + \, $}
\raisebox{-10pt}[0pt][0pt]{$\, \ddots \ $}
\raisebox{-12pt}[0pt][0pt]{$+ \, \cfrac{1}{m_k}$}
}},}$}
\end{picture}
\end{center}
where $k \ge 1$, $(m_1, \dots, m_k) \in (\mathbb{Z}_+)^k$ and
$m_k \ge 2$ unless $k=1$. For brevity, we write $m$ for $m_1$.

\begin{lemma} [{\cite[Proposition ~4.3]{lee_sakuma}}]
\label{lem:properties}
The following hold.
\begin{enumerate}[\indent \rm (1)]
\item Suppose $k=1$, i.e., $r=1/m$.
Then $S(r)=(m,m)$.

\item Suppose $k\ge 2$. Then each term of $S(r)$ is either $m$ or $m+1$,
and $S(r)$ begins with $m+1$ and ends with $m$.
Moreover, the following hold.

\begin{enumerate}[\rm (a)]
\item If $m_2=1$, then no two consecutive terms of $S(r)$ can be $(m, m)$,
so there is a sequence of positive integers $(t_1,t_2,\dots,t_s)$ such that
\[
S(r)=(t_1\langle m+1\rangle, m, t_2\langle m+1\rangle, m, \dots,
t_s\langle m+1\rangle, m).
\]
Here, the symbol ``$t_i\langle m+1\rangle$'' represents $t_i$ successive $m+1$'s.

\item If $m_2 \ge 2$, then no two consecutive terms of $S(r)$ can be $(m+1, m+1)$,
so there is a sequence of positive integers $(t_1,t_2,\dots,t_s)$ such that
\[
S(r)=(m+1, t_1\langle m\rangle, m+1, t_2\langle m\rangle,
\dots,m+1, t_s\langle m\rangle).
\]
Here, the symbol ``$t_i\langle m\rangle$'' represents $t_i$ successive $m$'s.
\end{enumerate}
\end{enumerate}
\end{lemma}

\begin{definition}
\label{def_T(r)}
{\rm
If $k\ge 2$, the symbol $T(r)$ denotes the sequence
$(t_1,t_2,\dots,t_s)$ in Lemma ~\ref{lem:properties},
which is called the {\it $T$-sequence of slope $r$}.
The symbol $CT(r)$ denotes the cyclic
sequence represented by $T(r)$, which is called the
{\it cyclic $T$-sequence of slope $r$}.
}
\end{definition}

\begin{lemma} [{\cite[Proposition ~4.4 and Corollary ~4.6]{lee_sakuma}}]
\label{lem:induction1}
Let $\tilde{r}$ be the rational number defined as
\[
\tilde{r}=
\begin{cases}
[m_3, \dots, m_k] & \text{if $m_2=1$};\\
[m_2-1, m_3, \dots, m_k] & \text{if $m_2 \ge 2$}.
\end{cases}
\]
Then we have $CS(\tilde{r})=CT(r)$.
\end{lemma}

\begin{lemma} [{\cite[Proposition ~4.5]{lee_sakuma}}]
\label{lem:sequence}
The sequence $S(r)$ has a decomposition $(S_1, S_2, S_1, S_2)$ which satisfies the following.
\begin{enumerate} [\indent \rm (1)]
\item Each $S_i$ is symmetric,
i.e., the sequence obtained from $S_i$ by reversing the order is
equal to $S_i$. {\rm (}Here, $S_1$ is empty if $k=1$.{\rm )}

\item Each $S_i$ occurs only twice in
the cyclic sequence $CS(r)$.

\item The subsequence $S_1$ begins and ends with $m+1$.

\item The subsequence $S_2$ begins and ends with $m$.
\end{enumerate}
\end{lemma}

\begin{lemma} [{\cite[Proof of Proposition ~4.5]{lee_sakuma}}]
\label{lem:relation}
Let $\tilde{r}$ be the rational number defined as in Lemma ~\ref{lem:induction1}.
Also let $S(\tilde{r})=(T_1, T_2, T_1, T_2)$ and $S(r) =(S_1, S_2, S_1, S_2)$
be decompositions described as in Lemma ~\ref{lem:sequence}.
Then the following hold.
\begin{enumerate} [\indent \rm (1)]
\item If $m_2=1$ and $k=3$, then $T_1=\emptyset$, $T_2=(m_3)$,
and $S_1 =(m_3\langle m+1 \rangle)$, $S_2 =(m)$.

\item If $m_2=1$ and $k\ge 4$, then
$T_1=(t_1, \dots, t_{s_1})$, $T_2=(t_{s_1+1}, \dots, t_{s_2})$, and
\[
\begin{aligned}
S_1
&=(t_1 \langle m+1 \rangle, m, t_2 \langle m+1 \rangle,
\dots, t_{s_1-1}\langle m+1 \rangle, m, t_{s_1}\langle m+1 \rangle), \\
S_2 &=(m, t_{s_1+1}\langle m+1\rangle,m, \dots, m, t_{s_2}\langle m+1\rangle, m).
\end{aligned}
\]

\item If $m_2\ge 2$ and
$k=2$, then $T_1=\emptyset$, $T_2=(m_2-1)$,
and $S_1 =(m+1)$, $S_2 =((m_2-1)\langle m \rangle)$.

\item If $m_2 \ge 2$ and $k\ge 3$, then $T_1=(t_1, \dots, t_{s_1})$,
$T_2=(t_{s_1+1}, \dots, t_{s_2})$, and
\[
\begin{aligned}
S_1 &=(m+1, t_{s_1+1}\langle m\rangle, m+1,
\dots, m+1, t_{s_2}\langle m\rangle, m+1),\\
S_2 &=(t_1\langle m\rangle, m+1,t_2\langle m\rangle, \dots,
t_{s_1-1}\langle m\rangle, m+1, t_{s_1}\langle m\rangle).
\end{aligned}
\]
\end{enumerate}
\end{lemma}

By Lemmas ~\ref{lem:properties} and \ref{lem:relation},
we easily obtain the following corollary.

\begin{corollary}
\label{cor:cor-to-relation}
Let $S(r)=(S_1,S_2,S_1,S_2)$ be as in Lemma ~\ref{lem:sequence}.
Then the following hold.
\begin{enumerate} [\indent \rm (1)]
\item If $m_2=1$, then $(m+1,m+1)$ appears in $S_1$.

\item If $m_2\ge 2$ and if $r \ne [m,2]=2/(2m+1)$,
then $(m,m)$ appears in $S_2$.
\end{enumerate}
\end{corollary}

\section{Small cancellation theory}
\label{sec:small_cancellation_theory}

Let $F(X)$ be the free group with basis $X$. A subset $R$ of $F(X)$
is said to be {\it symmetrized},
if all elements of $R$ are cyclically reduced and, for each $w \in R$,
all cyclic permutations of $w$ and $w^{-1}$ also belong to $R$.

\begin{definition}
{\rm Suppose that $R$ is a symmetrized subset of $F(X)$.
A nonempty word $b$ is called a {\it piece} if there exist distinct $w_1, w_2 \in R$
such that $w_1 \equiv bc_1$ and $w_2 \equiv bc_2$.
The small cancellation conditions $C(p)$ and $T(q)$,
where $p$ and $q$ are integers such that $p \ge 2$ and $q \ge 3$,
are defined as follows (see \cite{lyndon_schupp}).
\begin{enumerate}[\indent \rm (1)]
\item Condition $C(p)$: If $w \in R$
is a product of $t$ pieces, then $t \ge p$.

\item Condition $T(q)$: For $w_1, \dots, w_t \in R$
with no successive elements $w_i, w_{i+1}$
an inverse pair $(i$ mod $t)$, if $t < q$, then at least one of the products
$w_1 w_2,\dots,$ $w_{t-1} w_t$, $w_t w_1$ is freely reduced without cancellation.
\end{enumerate}
}
\end{definition}

We recall the following lemma from \cite{lee_sakuma},
which concerns the word $u_r$ defined in the beginning of
Section ~\ref{group_presentation}.

\begin{lemma} [{\cite[Lemma ~5.3]{lee_sakuma}}]
\label{maximal_piece}
Suppose that $r$ is a rational number with $0<r<1$,
and write $r=[m_1, m_2, \dots, m_k]$,
where $k \ge 1$, $(m_1, \dots, m_k) \in (\mathbb{Z}_+)^k$ and
$m_k \ge 2$.
Let $S(r)=(S_1, S_2, S_1, S_2)$ be as in Lemma ~\ref{lem:sequence}.
Decompose
\[
u_r \equiv v_1 v_2 v_3 v_4,
\]
where $S(v_1)=S(v_3)=S_1$ and $S(v_2)=S(v_4)=S_2$.
Then the following hold.
\begin{enumerate}[\indent \rm (1)]
\item If $k=1$, then the following hold.

\begin{enumerate}[\rm (a)]
\item No piece can contain $v_2$ or $v_4$.

\item No piece is of the form
$v_{2e} v_{4b}$ or $v_{4e} v_{2b}$,
where $v_{ib}$ and $v_{ie}$ are nonempty initial and terminal subwords of $v_i$, respectively.

\item Every subword of the form $v_{2b}$, $v_{2e}$, $v_{4b}$, or $v_{4e}$ is a piece,
where $v_{ib}$ and $v_{ie}$ are nonempty initial and terminal subwords of $v_i$ with $|v_{ib}|, |v_{ie}| \le |v_i|-1$, respectively.
\end{enumerate}

\item If $k \ge 2$, then the following hold.

\begin{enumerate}[\rm (a)]
\item No piece can contain $v_1$ or $v_3$.

\item No piece is of the form
$v_{1e} v_2 v_{3b}$ or $v_{3e} v_4 v_{1b}$,
where $v_{ib}$ and $v_{ie}$ are nonempty initial and terminal subwords of $v_i$, respectively.

\item Every subword of the form $v_{1e} v_2$, $v_2 v_{3b}$, $v_{3e} v_4$, or $v_4 v_{1b}$ is a piece,
where $v_{ib}$ and $v_{ie}$ are nonempty initial and terminal subwords of $v_i$ with $|v_{ib}|, |v_{ie}| \le |v_i|-1$, respectively.
\end{enumerate}
\end{enumerate}
\end{lemma}

By using the above lemma,
we establish the following key lemma concerning the cyclic word $(u_r^n)$,
where $u_r^n$ is the single relator of the presentation
$\Hecke(r;n)=\langle a, b \svert u_r^n \rangle$.

\begin{lemma}
\label{maximal_piece2}
Suppose that $r$ is a rational number with $0<r<1$,
and write $r=[m_1, m_2, \dots, m_k]$,
where $k \ge 1$, $(m_1, \dots, m_k) \in (\mathbb{Z}_+)^k$ and
$m_k \ge 2$.
Decompose $u_r \equiv v_1 v_2 v_3 v_4$ as in Lemma ~\ref{maximal_piece}.
Then for the relator $u_r^n \equiv (v_1 v_2 v_3 v_4)^n$,
where $n \ge 2$ is an integer,
the following hold.
\begin{enumerate}[\indent \rm (1)]
\item
The cyclic word $(u_r^n)$ is not a product of $t$ pieces with $t\le 4n-1$.
\item
Let $w$ be a subword of the cyclic word $(u_r^n)$
which is a product of $4n-1$ pieces
but is not a product of $t$ pieces with $t<4n-1$.
Then $w$ contains a subword, $w'$,
such that $S(w')=((2n-1)\langle S_1,S_2\rangle, \ell)$
or $S(w')=(\ell, (2n-1)\langle S_2,S_1\rangle)$,
where $S(r)=(S_1,S_2,S_1,S_2)$ and $\ell\in\ZZ_+$.
\end{enumerate}
\end{lemma}

\begin{proof}
{\rm
For simplicity, we prove the lemma when $k\ge 2$.
The case where $k=1$ is treated similarly.

(1) Let $(u_r^n)\equiv (w_1w_2\cdots w_t)$ be a decomposition
of the cyclic word $(u_r^n)$ into $t$ pieces.
Such a decomposition is determined by a $t$-tuple of ``breaks''
arranged in the cyclic word $(u_r^n)$,
such that $w_i$ is the subword of $(u_r^n)$
surrounded by the $(i-1)$-th break and the $i$-th break.
(Here the indices are considered modulo $t$.)
Then Lemma ~\ref{maximal_piece}(2-a) and (2-b) imply the following:
\begin{enumerate}[\indent \rm (a)]
\item
Each subword of the form $v_1$ or $v_3$ of $(u_r^n)$
contains a break in its interior.

\item
Each subword of the form $v_2$ or $v_4$ of $(u_r^n)$
contains a break in its interior or in its boundary.
\end{enumerate}
Since each break is contained in either
(a) the interior of a subword of the form $v_1$ or $v_3$
or
(b) the interior or the boundary of a subword of the form $v_2$ or $v_4$,
the above observation implies that
there is a well-defined surjection, $\eta$,
from the set of breaks onto the set
of subwords of the form
$v_1$, $v_2$, $v_3$ or $v_4$.
Since the domain and the codomain of $\eta$ have
cardinalities $t$ and $4n$, respectively,
we have $t\ge 4n$.
This completes the proof of assertion (1).
Before proving (2), we note that if $t$ is the smallest length
of decompositions of $(u_r^n)$ into pieces,
then Lemma ~\ref{maximal_piece}(2-c) implies that
$\eta$ is injective.

(2) Let $w\equiv w_1w_2\cdots w_{4n-1}$ be a subword of
the cyclic word $(u_r^n)$,
where $w_1, \cdots, w_{4n-1}$ are pieces,
such that $w$ is not a product of $t$ pieces with $t<4n-1$.
As in the proof of (1),
the decomposition $w\equiv w_1w_2\cdots w_{4n-1}$
is determined by a $(t+1)$-tuple of breaks in $(u_r^n)$,
such that $w_i$ is the subword of $(u_r^n)$
surrounded by the $(i-1)$-th break and the $i$-th break.
Lemma ~\ref{maximal_piece} implies the following:
\begin{enumerate}[\indent \rm (a)]
\item
Each subword of the form $v_1$ or $v_3$ of $(u_r^n)$
contains a unique break in its interior.

\item
Each subword of the form $v_2$ or $v_4$ of $(u_r^n)$
contains a unique break in its interior or in its boundary.
\end{enumerate}
Suppose first that the $0$-th break is contained in
the interior of a subword of $(u_r^n)$ of the form $v_1$.
Then we see from the above observations that
$w\equiv v_{1e}(v_2v_3v_4v_1)^{n-1}v_2v_3v_{4b}$,
where $v_{1e}$ is a nonempty proper terminal subword of $v_1$
and $v_{4b}$ is a
(possibly empty or nonproper)
initial subword of $v_4$.
Let $w'$ be the subword $v_{1e}'(v_2v_3v_4v_1)^{n-1}v_2v_3$
of $w$, where $v_{1e}'$ is a nonempty positive or negative
terminal subword of $v_{1e}$.
Then we have
$S(w')=(\ell, (2n-1)\langle S_2,S_1\rangle)$, where $\ell\in\ZZ_+$.
Suppose next that the $0$-th break is contained in
the interior or the boundary of a subword of $(u_r^n)$ of the form $v_2$.
Then we see from the above observations
$w\equiv v_{2e}(v_3v_4v_1v_2)^{n-1}v_3v_4v_{1b}$,
where $v_{2e}$ is a
(possibly empty or nonproper)
terminal subword of $v_2$
and $v_{1b}$ is a nonempty proper initial subword of
$v_1$.
Let $w'$ be the subword $(v_3v_4v_1v_2)^{n-1}v_3v_4v_{1b}'$ of $w$,
where $v_{1b}'$ is a non-empty initial
positive or negative subword of $v_{1b}$.
Then we have
$S(w')=((2n-1)\langle S_1,S_2\rangle,\ell)$, where $\ell\in\ZZ_+$.
The case where
the $0$-th break is contained in
the interior of a subword of $(u_r^n)$ of the form $v_3$
and the case where
$0$-th break is contained in
the interior or the boundary of a subword of $(u_r^n)$ of the form $v_4$
are treated similarly.
}
\end{proof}

The following proposition enables us to
apply small cancellation theory to our problem.

\begin{proposition}
\label{prop:small_cancellation_condition_heckoid}
Suppose that $r$ is a rational number with $0 < r< 1$
and that $n$ is an integer with $n \ge 2$.
Let $R$ be the symmetrized subset of $F(a, b)$ generated
by the single relator $u_{r}^n$ of the presentation $\Hecke(r;n)=\langle a, b \svert u_r^n \rangle$.
Then $R$ satisfies $C(4n)$ and $T(4)$.
\end{proposition}

\begin{proof}
{\rm
The assertion that $R$ satisfies $C(4n)$
is nothing other than Lemma ~\ref{maximal_piece2}(1).
The assertion that $R$ satisfies $T(4)$
is proved exactly as in \cite[Proof of Theorem ~5.1]{lee_sakuma}.
}
\end{proof}

Now we want to investigate the geometric consequences of
Proposition ~\ref{prop:small_cancellation_condition_heckoid}.
Let us begin with necessary definitions and notation following \cite{lyndon_schupp}.
A {\it map} $M$ is a finite $2$-dimensional cell complex embedded in $\RR^2$,
namely a finite collection of vertices ($0$-cells), edges ($1$-cells),
and faces ($2$-cells) in $\RR^2$.
The boundary (frontier), $\partial M$, of $M$ in $\RR^2$
is regarded as a $1$-dimensional subcomplex of $M$.
An edge may be traversed in either of two directions.
If $v$ is a vertex of a map $M$, then $d_M(v)$, the {\it degree of $v$}, will
denote the number of oriented edges in $M$ having $v$ as initial vertex.
A vertex $v$ of $M$ is called an {\it interior vertex}
if $v\not\in \partial M$, and an edge $e$ of $M$ is called
an {\it interior edge} if $e\not\subset \partial M$.

A {\it path} in $M$ is a sequence of oriented edges $e_1, \dots, e_t$ such that
the initial vertex of $e_{i+1}$ is the terminal vertex of $e_i$ for
every $1 \le i \le t-1$. A {\it cycle} is a closed path, namely
a path $e_1, \dots, e_t$ such that the initial vertex of $e_1$ is the terminal vertex of $e_n$.
If $D$ is a face of $M$, then any cycle of minimal length which includes
all the edges of the boundary, $\partial D$, of $D$
is called a {\it boundary cycle} of $D$.
By $d_M(D)$, the {\it degree of $D$}, we denote the number of
oriented edges in a boundary cycle of $D$.

\begin{definition}
{\rm A non-empty map $M$ is called a {\it $[p, q]$-map} if the following conditions hold.
\begin{enumerate}[\indent \rm (i)]
\item $d_M(v) \ge p$ for every interior vertex $v$ in $M$.

\item $d_M(D) \ge q$ for every face $D$ in $M$.
\end{enumerate}
}
\end{definition}

If $M$ is connected and simply connected, then a {\it boundary cycle} of $M$
is defined to be a cycle of minimal length which contains all the edges of
$\partial M$ going around once along the boundary of $\mathbb{R}^2 -M$.

\begin{definition}
{\rm Let $R$ be a symmetrized subset of $F(X)$. An {\it $R$-diagram} is
a map $M$ and a function $\phi$ assigning to each oriented edge $e$ of $M$, as a {\it label},
a reduced word $\phi(e)$ in $X$ such that the following hold.
\begin{enumerate}[\indent \rm (1)]
\item If $e$ is an oriented edge of $M$ and $e^{-1}$ is the oppositely oriented edge, then $\phi(e^{-1})=\phi(e)^{-1}$.

\item For any boundary cycle $\delta$ of any face of $M$,
$\phi(\delta)$ is a cyclically reduced word
representing an element of $R$.
(If $\alpha=e_1, \dots, e_t$ is a path in $M$, we define $\phi(\alpha) \equiv \phi(e_1) \cdots \phi(e_t)$.)
\end{enumerate}
In particular, if a group $G$ is presented by $G=\langle X \,|\, R \, \rangle$ with $R$ being symmetrized,
then a connected and simply connected $R$-diagram is called a {\it van Kampen diagram}
over the group presentation $G=\langle X \,|\, R \, \rangle$.
}
\end{definition}

Let $D_1$ and $D_2$ be faces (not necessarily distinct) of $M$
with an edge $e \subseteq \partial D_1 \cap \partial D_2$.
Let $e \delta_1$ and $\delta_2e^{-1}$ be boundary cycles of $D_1$ and $D_2$, respectively.
Let $\phi(\delta_1)=f_1$ and $\phi(\delta_2)=f_2$.
An $R$-diagram $M$ is called {\it reduced} if one never has $f_2=f_1^{-1}$.
It should be noted that if $M$ is reduced
then $\phi(e)$ is a piece for every interior edge $e$ of $M$.
A {\it boundary label of $M$} is defined to be a word $\phi(\alpha)$ in $X$ for $\alpha$ a boundary cycle of $M$.
It is easy to see that any two boundary labels of $M$ are cyclic permutations of each other.

We recall the following lemma which is a well-known
classical result in combinatorial group theory
(see \cite{lyndon_schupp}).

\begin{lemma} [van Kampen]
\label{van_Kampen}
Suppose $G=\langle X \,|\, R \, \rangle$ with $R$ being symmetrized.
Let $v$ be a cyclically reduced
word in $X$. Then $v=1$ in $G$ if and only if
there exists a reduced van Kampen diagram $M$
over $G=\langle X \,|\, R \, \rangle$
with a boundary label $v$.
\end{lemma}

As explained in \cite[Convention ~1] {lee_sakuma},
we may assume the following convention.

\begin{convention}
\label{convention}
{\rm Let $R$ be the symmetrized subset of $F(a, b)$
generated by the single relator $u_r^n$ of the presentation
$\Hecke(r;n)=\langle a, b \svert u_r^n \rangle$.
For any reduced $R$-diagram $M$,
we assume that $M$ satisfies the following.
\begin{enumerate}[\indent \rm (1)]
\item Every interior vertex of $M$ has degree at least three.

\item For every edge $e$ of $\partial M$,
the label $\phi(e)$ is a piece.

\item For a path $e_1, \dots, e_t$ in $\partial M$ of length $n\ge 2$
such that the vertex $e_i\cap e_{i+1}$ has degree $2$
for $i=1,2,\dots, t-1$,
$\phi(e_1) \phi(e_2) \cdots\phi(e_t)$ cannot be expressed as a product of less than $t$ pieces.
\end{enumerate}
}
\end{convention}

The following corollary is immediate from Proposition ~\ref{prop:small_cancellation_condition_heckoid}
and Convention ~\ref{convention}.

\begin{corollary}
\label{small_cancellation_condition_2}
Suppose that $r$ is a rational number with $0 < r< 1$
and that $n$ is an integer with $n \ge 2$.
Let $R$ be the symmetrized subset of $F(a, b)$
generated by the single relator $u_r^n$ of the presentation $\Hecke(r;n)=\langle a, b \svert u_r^n \rangle$.
Then every reduced $R$-diagram is a $[4, 4n]$-map.
\end{corollary}

We recall the following lemma
obtained from the arguments of \cite[Theorem ~V.3.1]{lyndon_schupp}.

\begin{lemma}[cf. {\cite[Theorem ~V.3.1]{lyndon_schupp}}]
\label{lem:inequality}
Let $M$ be an arbitrary connected and simply-connected map. Then
\[
4 \le \sum_{v \in \partial M} (3-d_M(v))+ \sum_{v \in M -\partial M} (4-d_M(v))+ \sum_{D \in M} (4-d_M(D)).
\]
In particular, if $M$ is a $[4,4n]$-map, then
\[
4 \le \sum_{v \in \partial M} (3-d_M(v)) + \sum_{D \in M} (4-4n).
\]
\end{lemma}

We now close this section with the following proposition
which will play an important role
in the proof of Theorem ~\ref{thm:null-homotopy}.

\begin{proposition}
\label{prop:key}
Let $M$ be an arbitrary connected and simply-connected $[4,4n]$-map
such that there is no vertex of degree $3$ in $\partial M$.
Put
\[
\begin{aligned}
A= &\ \text{\rm the number of vertices $v$ in $\partial M$ such that $d_M(v)=2$,} \\
B= &\ \text{\rm the number of vertices $v$ in $\partial M$ such that $d_M(v) \ge 4$.}
\end{aligned}
\]
Then the following inequality holds.
\[
A \ge (4n-3)B+4n
\]
\end{proposition}

\begin{proof}
{\rm
Put
\[
\begin{aligned}
V= &\ \text{\rm the number of vertices of $M$,} \\
E= &\ \text{\rm the number of (unoriented) edges of $M$,} \\
F= &\ \text{\rm the number of faces of $M$.}
\end{aligned}
\]
Then, since every interior vertex in $M$ has degree at least $4$,
we have
\[E \ge {1 \over 2}\{2A+4(V-A)\}=2V-A.\]
This inequality together with Euler's formula $1=V-E+F$ yields
$1 \le V-(2V-A)+F$, so that
\[\tag{\dag}
F \ge V-A+1 \ge (A+B)-A+1=B+1.
\]
On the other hand, by Lemma ~\ref{lem:inequality}, we have
\[
4 \le \sum_{v \in \partial M} (3-d_M(v)) + \sum_{D \in M} (4-4n)
=\sum_{v \in \partial M} (3-d_M(v)) + (4-4n)F,
\]
so that $\sum_{v \in \partial M} (3-d_M(v)) \ge 4+ (4n-4)F$.
Here, since $A-B \ge \sum_{v \in \partial M} (3-d_M(v))$
and since $(4n-4)F \ge (4n-4)(B+1)$ by ($\dag$),
we have
\[
A-B \ge (4n-4)(B+1)+4=(4n-4)B+4n,
\]
so that $A \ge (4n-4)B+4n+B=(4n-3)B+4n$, as required.
}
\end{proof}

\begin{corollary}
\label{cor:key}
Let $r$ be a rational number with $0 < r< 1$
and let $n$ be an integer with $n \ge 2$.
Write $r=[m_1, m_2, \dots, m_k]$, where $k \ge 1$,
$(m_1, \dots, m_k) \in (\mathbb{Z}_+)^k$ and $m_k \ge 2$,
and let $S(r)=(S_1, S_2, S_1, S_2)$ be as in Lemma ~\ref{lem:sequence}.
Suppose that $v$ is a
cyclically alternating word
which represents the trivial element
in $\Hecke(r;n)=\langle a, b \svert u_r^n \rangle$.
Then the cyclic word $(v)$
contains a subword $w$
of the cyclic word $(u_r^{\pm n})$
which is a product of $4n-1$ pieces
but is not a product of less than $4n-1$ pieces.
In particular,
the cyclic $S$-sequence $CS(v)$
of the cyclic word $(v)$
satisfies the following conditions.
\begin{enumerate}[\indent \rm (1)]
\item
If $k=1$, then
$CS(v)$ contains
$((2n-2) \langle m_1 \rangle)$ as a subsequence.
\item
If $k\ge 2$, then $CS(v)$ contains
$((2n-1) \langle S_1, S_2 \rangle)$
or $((2n-1) \langle S_2, S_1 \rangle)$ as a subsequence.
\end{enumerate}
\end{corollary}

\begin{proof}
By Lemma ~\ref{van_Kampen},
there is a reduced connected and simply-connected
diagram $M$ over $\Hecke(r;n)=\langle a, b \svert u_r^n \rangle$ with
$(\phi(\partial M))=(v)$.
By Corollary ~\ref{small_cancellation_condition_2},
$M$ is a $[4,4n]$-map over $\Hecke(r;n)=\langle a, b \svert u_r^n \rangle$.
Furthermore, since $(\phi(\partial M))=(v)$ is cyclically alternating,
there is no vertex of degree $3$ in $\partial M$.
Then by Proposition ~\ref{prop:key},
we have $A \ge (4n-3)B+4n$,
where $A$ and $B$ denote the numbers of vertices $v$ in $\partial M$ such that $d_M(v)=2$
and $d_M(v) \ge 4$, respectively.
This implies that there are at least $4n-2$ consecutive vertices of degree $2$ on $\partial M$.
Hence, by Convention ~\ref{convention},
the cyclic word $(\phi(\partial M))=(v)$ contains a subword $w$
of the cyclic word $(u_r^{\pm n})$
which is a product of $4n-1$ pieces
but is not a product of less than $4n-1$ pieces.
By Lemma ~\ref{maximal_piece2}(2), we may assume that
$S(w)=((2n-1) \langle S_1, S_2 \rangle, \ell)$ or $S(w)=(\ell, (2n-1) \langle S_2, S_1 \rangle)$,
where $\ell \in \ZZ_+$.
It follows that if $k=1$, then $CS(v)$ contains
$((2n-2) \langle m_1 \rangle)$ as a subsequence,
while if $k\ge 2$, then $CS(v)$ contains $((2n-1) \langle S_1, S_2 \rangle)$
or $((2n-1) \langle S_2, S_1 \rangle)$ as a subsequence.
\end{proof}

\begin{remark}
\rm
{In \cite[Theorem ~3]{Newman} (cf. \cite[Theorem ~IV.5.5]{lyndon_schupp}),
Newman gives a powerful theorem
for the word problem for one relator groups with torsion,
which implies that if a cyclically reduced word $v$
represents the trivial element in
$\Hecke(r;n)\cong \langle a,b \, |\, u_r^n\rangle$,
then the cyclic word $(v)$
contains a subword of the cyclic word $(u_r^{\pm n})$
of length greater than $(n-1)/n=1-1/n$ times the length of $u_r^n$.
Though
the above Corollary ~\ref{cor:key} is applicable only
when $v$ is cyclically alternating,
it imposes a stronger restriction on $(v)$.
In fact,
since every piece has length less than a half of the length of $u_r$
(see Lemma ~\ref{maximal_piece}),
Corollary \ref{cor:key} implies that
such a cyclic word $(v)$ contains
a subword of the cyclic word $(u_r^{\pm n})$
of length greater than
$1-1/(2n)$
times the length of $u_r^n$.
}
\end{remark}

\section{Proof of Theorem ~\ref{thm:null-homotopy}}
\label{sec:proof_of_the_theorem}

Throughout this section, suppose that $r$ is a rational number with $0<r<1$,
write $r=[m_1, m_2, \dots, m_k]$, where $k \ge 1$,
$(m_1, \dots, m_k) \in (\mathbb{Z}_+)^k$ and $m_k \ge 2$,
and let $n$ be an integer with $n \ge 2$.
Recall that the region, $R$, bounded by a pair of
Farey edges with an endpoint $\infty$
and a pair of Farey edges with an endpoint $r$
forms a fundamental domain for the action of $\RGPC{r}{n}$ on $\HH^2$
(see Figure ~\ref{fig.fd_orbifold}).
Let $I_1(r;n)$ and $I_2(r;n)$ be the (closed or half-closed) intervals in $\RR$
defined as follows:
\[
\begin{aligned}
I_1(r;n) &=
\begin{cases}
[0, r_1), \ \mbox{where} \ r_1=[m_1, \dots, m_k, 2n-2], & \mbox{if $k$ is odd,}\\
[0, r_1], \ \mbox{where} \ r_1=[m_1, \dots, m_{k-1}, m_k-1, 2] , & \mbox{if $k$ is even,}
\end{cases}\\
I_2(r;n) &=
\begin{cases}
[r_2, 1], \ \mbox{where} \ r_2=[m_1, \dots, m_{k-1}, m_k-1, 2], & \mbox{if $k$ is odd,}\\
(r_2,1], \ \mbox{where} \ r_2=[m_1, \dots, m_k, 2n-2], & \mbox{if $k$ is even.}
\end{cases}
\end{aligned}
\]
Then we may choose a fundamental domain $R$ so that
the intersection of $\bar R$ with $\partial \HH^2$ is equal to
the union $\bar I_1(r;n) \cup \bar I_2(r;n)\cup \{\infty,r\}$.

\begin{proposition}
\label{prop:connection}
Let $S(r)= (S_1, S_2, S_1, S_2)$ be as in Lemma ~\ref{lem:sequence}.
Then, for any $0 \neq s \in I_1(r;n) \cup I_2(r;n)$, the following hold.
\begin{enumerate}[\indent \rm (1)]
\item If $k=1$, that is, $r=1/m=[m]$, then $CS(s)$ does not contain
$((2n-2) \langle m \rangle)$ as a subsequence.

\item If $k \ge 2$, then $CS(s)$ does not contain
$((2n-1) \langle S_1, S_2 \rangle)$ nor $((2n-1) \langle S_2, S_1 \rangle)$ as a
subsequence.
\end{enumerate}
\end{proposition}

In the above proposition, we mean by a {\it subsequence} a subsequence without leap.
Namely a sequence $(a_1,a_2,\dots, a_p)$
is called a {\it subsequence} of a cyclic sequence,
if there is a sequence $(b_1,b_2,\dots, b_t)$
representing the cyclic sequence
such that $p\le t$ and $a_i=b_i$ for $1\le i\le p$.

\begin{proof}
{\rm
(1) Suppose that $r=1/m=[m]$.
Then any rational number $0 \neq s \in I_1(r;n) \cup I_2(r;n)=[0,r_1) \cup [r_2, 1]$,
where $r_1=(2n-2)/((2n-2)m+1)=[m, 2n-2]$ and $r_2=2/(2m-1)=[m-1, 2]$,
has a continued fraction expansion $s=[l_1, \dots, l_t]$,
where $t \ge 1$, $(l_1, \dots, l_t) \in (\mathbb{Z}_+)^t$ and $l_t \ge 2$ unless $t=1$,
such that
\begin{enumerate}[\indent \rm (i)]
\item $t \ge 1$ and $1 \le l_1 \le m-2$;

\item $t=1$ and $l_1=m-1$;

\item $t \ge 2$, $l_1=m-1$ and $l_2 \ge 2$;

\item $t \ge 3$, $l_1=m$ and $l_2=1$;

\item $t \ge 2$, $l_1=m$ and $2 \le l_2 \le 2n-3$; or

\item $t \ge 1$ and $l_1 \ge m+1$.
\end{enumerate}
If (i) happens, then $s=[l_1, l_2, \dots, l_t]$ with $1 \le l_1 \le m-2$,
so each component of $CS(s)$ is equal to $l_1 \le m-2$
or $l_1+1\le m-1$ by Lemma ~\ref{lem:properties}. Hence the assertion holds.
If (ii) happens, then $s=[m-1]$,
so $CS(s)= \lp m-1, m-1 \rp$. Hence the assertion holds.
If (iii) happens, then $s=[m-1, l_2, \dots, l_t]$ with $l_2 \ge 2$,
so $CS(s)$ consists of $m-1$ and $m$ but it does not have $(m,m)$ as a subsequence
by Lemma ~\ref{lem:properties}. Hence the assertion holds.
If (iv) happens, then $s=[m, 1, l_3, \dots, l_t]$,
so $CS(s)$ consists of $m$ and $m+1$ but it does not have $(m,m)$ as a subsequence
by Lemma ~\ref{lem:properties}. Hence the assertion holds.
If (v) happens, then $s=[m, l_2, \dots, l_t]$ with $2 \le l_2 \le 2n-3$,
so $CS(s)$ consists of $m$ and $m+1$ by Lemma ~\ref{lem:properties}.
Also by Lemma ~\ref{lem:induction1},
$\tilde{s}=[l_2-1, l_3, \dots, l_t]$ and $CS(\tilde{s})=CT(s)$.
Again by Lemma ~\ref{lem:properties}, each component of $CS(\tilde{s})=CT(s)$
is equal to $l_2-1 \le 2n-4$ or $l_2 \le 2n-3$.
This implies by Definition ~\ref{def_T(r)} that
$CS(s)$ contains at most $((2n-3) \langle m \rangle)$ as a subsequence, as required.
Finally, if (vi) happens, then $s=[l_1, l_2, \dots, l_t]$ with $l_1 \ge m+1$,
so each component of $CS(s)$ is equal to $l_1 \ge m+1$ or $l_1+1\ge m+2$
by Lemma ~\ref{lem:properties}. Hence the assertion holds.

(2) The proof proceeds by induction on $k \ge 2$.
For simplicity, we write $m$ for $m_1$.
By Lemma ~\ref{lem:sequence}, $S_1$ begins and ends with $m+1$,
and $S_2$ begins and ends with $m$.
Suppose on the contrary that there exists some $0 \neq s \in I_1(r;n) \cup I_2(r;n)$
for which $CS(s)$ contains $((2n-1) \langle S_1, S_2 \rangle)$ or $((2n-1) \langle S_2, S_1 \rangle)$
as a subsequence. This implies by Lemma ~\ref{lem:properties} that
$CS(s)$ consists of $m$ and $m+1$.
So $s$ has a continued fraction expansion $s=[l_1, \dots, l_t]$, where
$t \ge 2$, $(l_1, \dots, l_t) \in (\mathbb{Z}_+)^t$, $l_1=m$ and $l_t \ge 2$.
For the rational numbers $r$ and $s$, define the rational numbers
$\tilde{r}$ and $\tilde{s}$ as in Lemma ~\ref{lem:induction1}
so that $CS(\tilde{r})=CT(r)$ and $CS(\tilde{s})=CT(s)$.

We consider three cases separately.

\medskip
\noindent {\bf Case 1.} {\it $m_2=1$.}
\medskip

In this case, $k \ge 3$ and, by Corollary ~\ref{cor:cor-to-relation}(1),
$(m+1, m+1)$ appears in $S_1$ as a subsequence,
so in $CS(s)$ as a subsequence.
Thus by Lemma ~\ref{lem:properties}, $l_2=1$ and so $t \ge 3$.
So, we have
\[
\tilde{r}=[m_3, \dots, m_k] \quad \text{\rm and} \quad
\tilde{s}=[l_3, \dots, l_t].
\]
It follows from $0 \neq s \in I_1(r;n) \cup I_2(r;n)$ that
$0 \neq \tilde{s} \in I_1(\tilde{r};n) \cup I_2(\tilde{r};n)$.
At this point, we divide this case into two subcases.

\medskip
\noindent {\bf Case 1.a.} {\it $k=3$.}
\medskip

By Lemma ~\ref{lem:relation}(1), $S_1 =(m_3\langle m+1 \rangle)$ and $S_2 =(m)$.
Since $((2n-1) \langle S_1, S_2 \rangle)$ or $((2n-1) \langle S_2, S_1 \rangle)$
is contained in $CS(s)$ by assumption, $(S_2, (2n-2) \langle S_1, S_2 \rangle)$
is contained in $CS(s)$. This implies that $CS(\tilde{s})=CT(s)$ contains
$((2n-2) \langle m_3 \rangle)$ as a subsequence.
But since $\tilde{r}=1/m_3=[m_3]$ and $0 \neq \tilde{s} \in I_1(\tilde{r};n) \cup I_2(\tilde{r};n)$,
this gives a contradiction to (1).

\medskip
\noindent {\bf Case 1.b.} {\it $k\ge 4$.}
\medskip

Let $S(\tilde{r})=(T_1, T_2, T_1, T_2)$ be the decomposition of $S(\tilde{r})$
given by Lemma ~\ref{lem:sequence}.
Since $S_1$ begins and ends with $m+1$, $S_2$ begins and ends with $m$,
and since $((2n-1) \langle S_1, S_2 \rangle)$ or $((2n-1) \langle S_2, S_1 \rangle)$
is contained in $CS(s)$ by assumption,
we see by Lemma ~\ref{lem:relation}(2) that
$CS(\tilde{s})=CT(s)$ contains, as a subsequence,
\[
\begin{aligned}
&(t_1+\ell', t_2, \dots, t_{s_1-1}, t_{s_1}, T_2, (2n-2) \langle T_1, T_2 \rangle), \ \mbox{or} \\
&((2n-2) \langle T_2, T_1 \rangle, T_2, t_1, t_2, \dots, t_{s_1-1}, t_{s_1}+\ell''),
\end{aligned}
\]
where $(t_1, t_2, \dots, t_{s_1})=T_1$ and $\ell', \ell'' \in\ZZ_+\cup\{0\}$.
(Note that
$((2n-1) \langle S_1, S_2 \rangle)$ begins with $m+1$ and ends with $m$,
whereas
$((2n-1) \langle S_2, S_1 \rangle)$ begins with $m$ and ends with $m+1$.)
Since $t_1=t_{s_1}=m_3+1$ by Lemma ~\ref{lem:sequence},
this actually implies that
$\ell'=0$ or $\ell''=0$ accordingly, and therefore
$CS(\tilde{s})$ contains
$((2n-1) \langle T_1, T_2 \rangle)$ or $((2n-1) \langle T_2, T_1 \rangle)$ as a subsequence.
But since $\tilde{r}=[m_3, \dots, m_k]$ and $0 \neq \tilde{s} \in I_1(\tilde{r};n) \cup I_2(\tilde{r};n)$,
this gives a contradiction to the induction hypothesis.

\medskip
\noindent {\bf Case 2.} {\it $k=2$ and $m_2=2$.}
\medskip

In this case, $r=[m,2]$, so by Lemma ~\ref{lem:relation}(3), $S_1=(m+1)$ and $S_2=(m)$.
Since $((2n-1) \langle S_1, S_2 \rangle)$ or $((2n-1) \langle S_2, S_1 \rangle)$
is contained in $CS(s)$ by assumption, both $(m+1, (2n-2) \langle m, m+1 \rangle)$
and $((2n-2) \langle m, m+1 \rangle, m)$ are contained in $CS(s)$.
This implies that $CS(\tilde{s})=CT(s)$ contains
$((2n-2) \langle 1 \rangle)$ as a subsequence.
Moreover, we can see that this subsequence is proper,
i.e., it is not equal to the whole cyclic sequence $CS(\tilde{s})=CT(s)$.
As described below, this in turn implies that
$s$ has the form
either
$s=[m,1, 1, l_4 \dots, l_t]$
or $s=[m,2, l_3, \dots, l_t]$ with $l_3 \ge 2n-2$.
If $l_2=1$, then $\tilde s=[l_3,\dots,l_t]$ and so
$l_3$ is the minimal component of $CS(\tilde s)$
(see Lemma ~\ref{lem:properties}).
Hence we must have $l_3=1$,
i.e.,
$s=[m,1, 1, l_4 \dots, l_t]$,
because $CS(\tilde{s})$ contains $1$ as a component.
On the other hand, if $l_2\ge 2$, then $\tilde s=[l_2-1,\dots,l_t]$ and so
$l_2-1$ is the minimal component of $CS(\tilde s)$
(see Lemma ~\ref{lem:properties}).
Since $CS(\tilde{s})$ contains $1$ as a component,
we have
$l_2-1=1$, i.e., $l_2=2$.
Since $CS(\tilde{s})$ contains $((2n-2) \langle 1 \rangle)$ as a
subsequence,
we see that
$CS(\tilde{\tilde s})=CT(\tilde s)$ contains a component $\ge 2n-2$.
Since the subsequence $((2n-2) \langle 1 \rangle)$ of
$CS(\tilde{s})$ is proper, we see $t\ge 3$ and $l_3\ge 2$.
Thus $\tilde{\tilde s}=[l_3-1,\dots,l_t]$ and therefore
$l_3-1$ is the minimal component of $CS(\tilde{\tilde s})$.
Hence we must have
$l_3=(l_3-1)+1\ge 2n-2$
and so $s=[m,2, l_3, \dots, l_t]$ with $l_3 \ge 2n-2$.

But then $s$ cannot belong to the interval
$I_1(r;n) \cup I_2(r;n)=[0,r_1] \cup (r_2, 1]$,
where $r_1=[m,1, 2]$ and $r_2=[m,2, 2n-2]$, a contradiction to the hypothesis.

\medskip
\noindent {\bf Case 3.} {\it Either both $k=2$ and $m_2 \ge 3$ or both $k \ge 3$ and $m_2 \ge 2$.}
\medskip

In this case, by Corollary ~\ref{cor:cor-to-relation}(2), $(m,m)$ appears in $S_2$ as a subsequence,
so in $CS(s)$ as a subsequence.
Thus $l_2\ge 2$ by Lemma ~\ref{lem:properties}, and so we have
\[
\tilde{r}=[m_2-1,m_3, \dots, m_k] \quad \text{\rm and} \quad
\tilde{s}=[l_2-1,l_3, \dots, l_t].
\]
It follows from $0 \neq s \in I_1(r;n) \cup I_2(r;n)$ that
$0 \neq \tilde{s} \in I_1(\tilde{r};n) \cup I_2(\tilde{r};n)$.
At this point, we consider two subcases separately.

\medskip
\noindent {\bf Case 3.a.} {\it $k=2$ and $m_2 \ge 3$.}
\medskip

By Lemma ~\ref{lem:relation}(3), $S_1=(m+1)$ and $S_2 =((m_2-1)\langle m \rangle)$.
Since $((2n-1) \langle S_1, S_2 \rangle)$ or $((2n-1) \langle S_2, S_1 \rangle)$
is contained in $CS(s)$ by assumption, $(S_1, (2n-2) \langle S_2, S_1 \rangle)$
is contained in $CS(s)$. This implies that $CS(\tilde{s})=CT(s)$ contains
$((2n-2) \langle m_2-1 \rangle)$ as a subsequence.
But since $\tilde{r}=1/(m_2-1)=[m_2-1]$ and $0 \neq \tilde{s} \in I_1(\tilde{r};n) \cup I_2(\tilde{r};n)$,
this gives a contradiction to (1).

\medskip
\noindent {\bf Case 3.b.}
{\it $k \ge 3$
and $m_2 \ge 2$.}
\medskip

Let $S(\tilde{r})= (T_1, T_2, T_1, T_2)$ be
the decomposition of $S(\tilde{r})$ given by Lemma ~\ref{lem:sequence}.
Since $S_1$ begins and ends with $m+1$, $S_2$ begins and ends with $m$,
and since $((2n-1) \langle S_1, S_2 \rangle)$ or $((2n-1) \langle S_2, S_1 \rangle)$
is contained in $CS(s)$ by assumption,
we see by Lemma ~\ref{lem:relation}(4) that
$CS(\tilde{s})=CT(s)$ contains, as a subsequence,
\[
\begin{aligned}
&((2n-2) \langle T_2, T_1 \rangle, T_2, t_1, t_2, \dots, t_{s_1-1}, t_{s_1}+\ell'), \ \mbox{or} \\
&(t_1+\ell'', t_2, \dots, t_{s_1-1}, t_{s_1}, T_2, (2n-2) \langle T_1, T_2 \rangle),
\end{aligned}
\]
where $(t_1, t_2, \dots,t_{s_1})=T_1$ and $\ell', \ell'' \in\ZZ_+\cup\{0\}$.
Since $t_1=t_{s_1}=(m_2-1)+1=m_2$ by Lemma ~\ref{lem:sequence},
this actually implies that
$\ell'=0$ or $\ell''=0$ accordingly,
and therefore
$CS(\tilde{s})$ contains
$((2n-1) \langle T_1, T_2 \rangle)$ or $((2n-1) \langle T_2, T_1 \rangle)$ as a subsequence.
But since $\tilde{r}=[m_2-1, m_3, \dots, m_k]$ and $0 \neq \tilde{s} \in I_1(\tilde{r};n) \cup I_2(\tilde{r};n)$,
this gives a contradiction to the induction hypothesis.

The proof of Proposition ~\ref{prop:connection} is completed.
}
\end{proof}

We are now in a position to prove Theorem ~\ref{thm:null-homotopy}.

\begin{proof}[Proof of Theorem ~\ref{thm:null-homotopy}]
Suppose on the contrary that there exists a rational number
$s\in I(r;n)\cup\{r\}= I_1(r;n) \cup I_2(r;n) \cup\{r\}$
for which $\alpha_s$ is null-homotopic in $\orbs(r;n)$.
Then $u_s$ equals the identity in $\Hecke(r;n)$.
Since $u_r$ is a non-trivial torsion element in
$\Hecke(r;n)=\langle a, b \svert u_r^n \rangle$ by
\cite[Theorem ~IV.5.2]{lyndon_schupp},
we may assume $s\in I_1(r;n) \cup I_2(r;n)$.
By Corollary ~\ref{cor:key},
the cyclic word $(u_s)$
contains a subword $w$
of the cyclic word $(u_r^{\pm n})$
which is a product of $4n-1$ pieces
but is not a product of less than $4n-1$ pieces.
Since
$4n-1 \ge 7$,
the length of such a subword $w$ is greater or equal to $7$.
So $s$ cannot be zero, because the word $u_0=ab$ cannot contain such a subword $w$.
By Corollary ~\ref{cor:key} again,
if $r=1/m$, then $CS(u_s)=CS(s)$ contains
$((2n-2) \langle m \rangle)$ as a subsequence,
while if $r \neq 1/m$, then $CS(s)$ contains $((2n-1) \langle S_1, S_2 \rangle)$
or $((2n-1) \langle S_2, S_1 \rangle)$ as a subsequence,
where $S(r)= (S_1, S_2, S_1, S_2)$ is as in Lemma ~\ref{lem:sequence}.
This contradicts
Proposition ~\ref{prop:connection}.
\end{proof}

\section*{Acknowledgement}
The authors would like to thank the referee
for very careful reading and helpful comments.

\bibstyle{plain}

\end{document}